\documentclass[12pt]{amsart}

\usepackage{amsmath,amssymb,amsthm,amsfonts}
\usepackage[T1]{fontenc}
\usepackage[utf8]{inputenc}
\usepackage{mathrsfs}
\usepackage[margin=1in]{geometry}
\usepackage{amscd}
\usepackage{latexsym}
\usepackage{color}
\usepackage{graphicx}
\usepackage{cite}
\usepackage[colorlinks=true]{hyperref}
\hypersetup{urlcolor=blue,citecolor=blue,linkcolor=blue}

\newtheorem{theorem}{Theorem}[section]
\newtheorem{corollary}[theorem]{Corollary}
\newtheorem{lemma}[theorem]{Lemma}
\newtheorem{proposition}[theorem]{Proposition}

\newtheorem{remark}[theorem]{Remark}

\theoremstyle{remark}
\numberwithin{equation}{section}

\newcommand{\diag}{\operatorname{diag}}
\newcommand{\T}{\operatorname{UT}}
\renewcommand{\t}{\mathfrak{ut}}
\newcommand{\R}{\mathbb{R}}
\newcommand{\w}{\mathbf{w}}
\renewcommand{\b}{\mathbf{b}}
\newcommand{\y}{\mathbf{y}}
\newcommand{\mtx}[1]{\begin{bmatrix}#1\end{bmatrix}}

\title{Weighted Power Means on Upper Triangular Matrix Groups}

\author{Huajun Huang}
\address{Department of Mathematics and Statistics, Auburn University, Auburn, AL 36849, USA}
\email{huanghu@auburn.edu}

\author{Jimmie Lawson}
\address{Department of Mathematics, Louisiana State University, Baton Rouge, LA 70803, USA}
\email{lawson@math.lsu.edu}

\begin{document}

\begin{abstract}
	We study weighted power means on the Lie group $\T_n$ of real upper
	triangular matrices with positive diagonal entries.  For
	$t\in\mathbb R\setminus\{0\}$, the weighted power mean
	$P_t(\mathbf{w};A_1,\ldots,A_m)$ is defined by
	\[
	X=\sum_{i=1}^m w_i(X\#_tA_i),
	\qquad
	X\#_tA:=X(X^{-1}A)^t.
	\]
	Using divided differences, analytic matrix functions, and the Parlett
	recurrence, we prove existence, uniqueness, and real-analytic dependence
	of $P_t$ on the data and the weights.  We derive finite entrywise and
	block-recursive algorithms, establish structural properties of the mean,
	and obtain explicit formulas for the cases  $\T_2$ and $\T_3$.	
	We prove that $P_t$ converges to the weighted Karcher mean as $t\to0$.
	We also establish a Lie--Trotter formula and derive second-order
	asymptotic expansions for the power, Karcher, and log-Euclidean means,
	showing that the parameter $t$ first enters the local behavior at second
	order.  \end{abstract}


\maketitle

\section{Introduction}
Statistics on manifolds (also called geometric statistics or non-Euclidean statistics) has been and continues to be a rich and growing area of research
and application in recent years.  Some manifolds that have arisen in applications include the cone of positive definite matrices, Lie groups
and homogeneous spaces, and hyperbolic space.  Various notions of a mean have appeared in these studies, a popular one being the
Karcher mean.  See \cite{Pennec} and \cite{PennecArsigny} for foundational ideas and examples.

In this paper we introduce another mean, the power mean, in the setting of Lie groups.
Power means constitute one of the basic one-parameter families of scalar
means.  If $a_1,\ldots,a_m>0$ and
$\w=(w_1,\ldots,w_m)$ is a probability vector, then
\[
M_t(\w;a_1,\ldots,a_m)
 =\left(\sum_{i=1}^m w_i a_i^t\right)^{1/t},\qquad t\ne0,
\]
and $M_t$ converges to the weighted geometric mean as $t\to0$.
Lim and P\'alfia \cite{LimPalfia2012} introduced a noncommutative analogue
on the cone of positive definite matrices.  For $0<t\le1$, their power mean
is the unique positive definite solution of
\[
X=\sum_{i=1}^m w_i(X\#_tA_i),
\]
where $\#_t$ is the weighted positive definite geometric mean; negative
powers are obtained by inversion duality.  Among other properties, they
proved the important result that these means converge to the 
weighted Karcher mean as $t\to0$.

The aim of this paper is to develop an analogous  power-mean theory on the solvable Lie
group
\[
\T_n=\{A\in M_n(\R): A\ \text{is upper triangular and}\ A_{ii}>0 \text{ for all }i\},
\]
whose Lie algebra $\t_n$ consists of all real upper triangular matrices.
The exponential map $\exp:\t_n\to\T_n$ is a real-analytic diffeomorphism,
with inverse $\log$.  In particular, every element of $\T_n$ has positive
spectrum, so real powers and logarithms are defined globally by analytic
functional calculus, as given by the usual matrix exponential and 
logarithm maps.

For $A,B\in\T_n$ and $s\in\R$, define the weighted geometric mean
\begin{equation}\label{eq:sharp}
 A\#_sB:=A^{1/2}\bigl(A^{-1/2}BA^{-1/2}\bigr)^sA^{1/2}.
\end{equation}
Equivalently,
\[
A\#_sB
 =A(A^{-1}B)^s.
\]
This interpolation satisfies
\[
A\#_0B=A,\qquad A\#_1B=B,\qquad
A\#_s(A\#_rB)=A\#_{sr}B,
\]
and, for $P,Q\in\T_n$,
\begin{equation}\label{eq:sharp-covariance}
 (PAQ)\#_s(PBQ)=P(A\#_sB)Q.
\end{equation}
   When $0\leq s\leq 1$, the weighted geometric mean 
\[
A\#_s B=G((1-s,s);A,B),
\]
is a special case of the general weighted Karcher mean $X=G(\w ; A_1,\ldots, A_m)$ defined as the unique solution \(X\in\T_n\) of the Karcher equation:
\begin{equation}\label{Karcher equation}
	\sum_{i=1}^m w_i \log(X^{-1}A_i)=0.
\end{equation}
In \cite{HuangLawson26}, the authors presented  a recursive  analytic formula  for the weighted Karcher mean on $\T_n$
and explored its properties including existence and uniqueness. 

Now for $t\in \R\setminus\{0\}$ we consider the fixed-point equation
\begin{equation}\label{eq:power-mean}
 X=\sum_{i=1}^m w_i(X\#_tA_i),
\end{equation}
which, after left multiplication by $X^{-1}$, is equivalent to
\begin{equation}\label{eq:normalized-power}
 \sum_{i=1}^m w_i(X^{-1}A_i)^t=I_n.
\end{equation}
The unique solution, whose existence is proved below, is called the \emph{weighted power mean}  and is denoted by
$P_t(\w;A_1,\ldots,A_m)$. Moreover, when $t\to 0$, 
the rescaled power equation
\[
\sum_{i=1}^m w_i\frac{(X^{-1}A_i)^t-I_n}{t}=0
\]
converges formally to the Karcher equation
\eqref{Karcher equation}.
We subsequently prove that the weighted  power means converge entrywise to the 
weighted Karcher mean as $t$ approaches $0$.

The formal resemblance with the positive definite theory conceals a
substantial geometric difference.  On $\T_n$, the Karcher and power mean
equations do not arise from the same Riemannian and order structures that
govern the positive definite cone.  In particular, $\T_n$ is neither a
self-adjoint cone nor an ordered symmetric space, so the L\"owner order and
the Thompson-metric contraction arguments available in the positive
definite setting have no direct analogue here.  Our approach instead
exploits the triangular structure itself.  Analytic functions of upper
triangular matrices admit explicit divided-difference representations, and
the Parlett recurrence shows that, once the lower superdiagonals have been
determined, the next unknown entry enters
\eqref{eq:normalized-power} linearly.   This observation leads to a finite recursive construction and is
the key mechanism behind existence, uniqueness, and analyticity of the
power mean on $\T_n$.

Our first main result establishes that, for every $t\ne0$, every strictly positive probability vector
$\w\in\Delta_m^\circ$, and every $A_1,\ldots,A_m\in\T_n$, the power
equation \eqref{eq:power-mean} has a unique solution in $\T_n$, depending
real-analytically on the matrices and the weights.  The triangular
structure yields finite entrywise and block-recursive algorithms, including
explicit formulas on $\T_2$ and $\T_3$.  We also establish several
structural properties, including commuting consistency, homogeneity,
permutation invariance, two-sided covariance, principal-block consistency,
and exact diagonal and determinant formulas.  
We further establish a Lie--Trotter formula showing that every
nonzero-order power mean has the same first-order linearization at the
identity, namely the weighted arithmetic mean on $\t_n$.  In particular,
\[
\lim_{s\to0}
P_t\bigl(\w;A_1^s,\ldots,A_m^s\bigr)^{1/s}
=
\exp\left(\sum_{i=1}^m w_i\log A_i\right).
\]
A second-order expansion shows that the parameter $t$ first enters at
quadratic order and yields a local comparison among the power, Karcher,
and log-Euclidean means.  

The paper is organized as follows.  Section~2 develops preliminary
properties of the power equation, including negative orders, structural
properties, and diagonal-order consequences.  
Section~3 recalls the divided-difference and Parlett formulas for analytic
functions of upper triangular matrices.  Section~4 proves existence,
uniqueness, and real-analytic dependence, while Section~5 develops
entrywise and block-recursive algorithms and explicit low-dimensional
formulas.  Section~6 proves convergence to the weighted Karcher mean as
$t\to0$.  Section~7 establishes the Lie--Trotter formula, second-order
asymptotics, and local comparisons with the Karcher and log-Euclidean
means.   Section~8 gives concluding remarks.

\section{Preliminaries and basic properties}

Throughout,
\[
\Delta_m^\circ
 :=\left\{(w_1,\ldots,w_m)\in(0,1)^m:
 \sum_{i=1}^m w_i=1\right\},
\qquad \w=(w_1,\ldots,w_m)\in\Delta_m^\circ.
\]
For $A=[a_{ij}]\in\t_n$, let $\diag(A)$ denote its diagonal part.  For
$1\le p\le q\le n$, $A[p:q]$ denotes the \emph{principal block} indexed by
$p,p+1,\ldots,q$, the block submatrix having upper left entry $A_{pp}$ and lower right entry $A_{qq}$.  Let $E_{ij}$ denote the 
standard matrix unit with $ij$-entry $1$ and all other entries $0$.

\subsection{Principal blocks and analytic functional calculus}
Principal diagonal blocks
 are compatible with the operations used below.
For $A,B\in\T_n$ and $1\le p\le q\le n$,
\begin{equation}\label{eq:block-product}
 (AB)[p:q]=A[p:q]B[p:q],
 \qquad
 A^{-1}[p:q]=(A[p:q])^{-1},
\end{equation}
where $A[p:q]$ denotes the block submatrix with upper left entry $A_{pp}$
and lower right entry $A_{qq}$.

The spectrum $\sigma(A)$ of  $A$ consists of its diagonal entries. 
\begin{lemma}\label{lem:block-functional-calculus}
Let $A\in\T_n$ and let $f$ be analytic on a neighborhood of $\sigma(A)$.
Then
\begin{equation}\label{eq:block-function}
 f(A)[p:q]=f(A[p:q]),\qquad 1\le p\le q\le n.
\end{equation}
In particular, restriction to a principal block commutes with real powers
and with the logarithm.
\end{lemma}

\begin{proof}
Restriction to a principal diagonal block is an algebra homomorphism on upper triangular
	matrices:
	\[
	  (AB)[p:q]=A[p:q]B[p:q],
	\] and similarly addition and scalar multiplication are preserved.
	Consequently,
	\[
	P(A)[p:q]=P(A[p:q])
	\]
	for every polynomial $P$, and the result for analytic $f$ follows by
	analytic functional calculus.  
	\end{proof}

\subsection{Diagonal and block equations}

The following observation is the starting point of all subsequent
recursions involving the power mean.

\begin{lemma}\label{lem:diagonal-block}
Suppose $t\in\mathbb{R}\setminus \{0\}$ and $X\in\T_n$ satisfies \eqref{eq:normalized-power}.  Then
\begin{equation}\label{eq:diagonal}
 X_{jj}=\left(\sum_{i=1}^m w_i(A_i)_{jj}^t\right)^{1/t},
 \qquad j=1,\ldots,n,
\end{equation}
or equivalently
\[
 \diag(X)^t=\sum_{i=1}^m w_i\diag(A_i)^t.
\]
Moreover, for $1\le p\le q\le n$,
\begin{equation}\label{eq:block-equation}
 \sum_{i=1}^m w_i
 \bigl(X[p:q]^{-1}A_i[p:q]\bigr)^t=I_{q-p+1}.
\end{equation}
Consequently, after uniqueness has been established,
\begin{equation}\label{eq:block-consistency}
 P_t(\w;A_1,\ldots,A_m)[p:q]
 =P_t\bigl(\w;A_1[p:q],\ldots,A_m[p:q]\bigr).
\end{equation}
\end{lemma}

\begin{proof}
The diagonal entries of $X^{-1}A_i$ are $(A_i)_{jj}/X_{jj}$.  Taking the
$(j,j)$-entry of \eqref{eq:normalized-power} gives
\[
 \sum_{i=1}^m w_i
 \left(\frac{(A_i)_{jj}}{X_{jj}}\right)^t=1,
\]
which is equivalent to \eqref{eq:diagonal}.  Restricting
\eqref{eq:normalized-power} to $[p:q]$ and using
\eqref{eq:block-product}--\eqref{eq:block-function} yields
\eqref{eq:block-equation}.
\end{proof}

The diagonal variables are completely decoupled: each is the scalar weighted
power mean of the corresponding diagonal entry.  All genuinely
noncommutative information is therefore carried by the strictly upper
triangular part.

\subsection{Negative orders}

In the positive definite theory developed by Lim--P\'alfia \cite{LimPalfia2012}, the weighted power mean for negative exponent is defined by the inversion-duality convention: 
 for $t>0$,
\begin{equation}\label{eq:negative-intro}
	P_{-t}(\w;A_1,\ldots,A_m)
	:=
	P_t(\w;A_1^{-1},\ldots,A_m^{-1})^{-1}.
\end{equation}
Using the existence and uniqueness of the solution of  equation (\ref{eq:power-mean}), which will be established in Theorem \ref{thm:existence},  we show that this definition and 
\eqref{eq:normalized-power} coincide in $\T_n$.

\begin{proposition} 
	\label{prop:negative-power-mean}
	Let $t>0$ and let $A_1,\ldots,A_m\in\T_n$. Then the matrix
	\begin{equation}\label{eq:negative-power-definition}
		X 
		:=
		P_t(\w;A_1^{-1},\ldots,A_m^{-1})^{-1}.
	\end{equation}
	 is precisely the unique matrix
	$X\in\T_n$ satisfying
	\begin{equation}\label{eq:negative-power-equation}
		\sum_{i=1}^m w_i(X^{-1}A_i)^{-t}=I_n.
	\end{equation}
\end{proposition}

\begin{proof}
	Set
	\[
	Y=P_t(\w;A_1^{-1},\ldots,A_m^{-1}),
	\qquad
	X=Y^{-1}.
	\]
	Since $Y$ satisfies the positive-order normalized power equation,
	\[
	\sum_{i=1}^m w_i(Y^{-1}A_i^{-1})^t=I_n,
	\]
	we have
	\[
	\sum_{i=1}^m w_i(XA_i^{-1})^t=I_n.
	\]
	Now
	\[
	(X^{-1}A_i)^{-1}
	=
	A_i^{-1}X
	=
	X^{-1}(XA_i^{-1})X.
	\]
	By similarity equivariance of analytic functional calculus,
	\[
	(X^{-1}A_i)^{-t}
	=
	X^{-1}(XA_i^{-1})^tX.
	\]
	Hence
	\[
	\sum_{i=1}^m w_i(X^{-1}A_i)^{-t}
	=
	X^{-1}
	\left(
	\sum_{i=1}^m w_i(XA_i^{-1})^t
	\right)X
	=
	I_n.
	\]
	Thus the matrix in \eqref{eq:negative-power-definition} satisfies
	\eqref{eq:negative-power-equation}.
	
	Conversely, suppose that $X\in\T_n$ satisfies
	\eqref{eq:negative-power-equation}.  Using the same similarity identity,
	\[
	\sum_{i=1}^m w_i(XA_i^{-1})^t=I_n.
	\]
	Therefore $Y=X^{-1}$ satisfies
	\[
	\sum_{i=1}^m w_i(Y^{-1}A_i^{-1})^t=I_n.
	\]
	By uniqueness of the positive-order power mean,
	\[
	Y=P_t(\w;A_1^{-1},\ldots,A_m^{-1}),
	\]
	and hence
	\[
	X=P_t(\w;A_1^{-1},\ldots,A_m^{-1})^{-1}
	=
	P_{-t}(\w;A_1,\ldots,A_m).
	\]

\end{proof}

\subsection{Structural properties}

For brevity we write $\mathbf A=(A_1,\ldots,A_m)$.  The next proposition lists
the algebraic properties of the power mean that will be used later.  The   assertions follow
from the defining equation together with the uniqueness  proved in
Theorem~\ref{thm:existence}.

\begin{proposition}\label{prop:basic-properties}
Let $\mathbf A\in(\T_n)^m$, $\w\in\Delta_m^\circ$, and
$t\in\R \setminus\{0\}$.  Then the following hold.
\begin{enumerate}
\item[(i)] \emph{Commuting consistency.} If $A_1,\ldots,A_m$ commute
pairwise, then
\begin{equation}\label{eq:commuting}
 P_t(\w;\mathbf A)
 =\left(\sum_{i=1}^m w_iA_i^t\right)^{1/t}.
\end{equation}

\item[(ii)] \emph{Joint homogeneity.} For $a_1,\ldots,a_m>0$, set
\[
 c_t=\left(\sum_{i=1}^m w_i a_i^t\right)^{1/t},
 \qquad
 v_i=\frac{w_i a_i^t}{\sum_{j=1}^m w_j a_j^t}.
\]
Then
\begin{equation}\label{eq:joint-homogeneity}
 P_t(\w;a_1A_1,\ldots,a_mA_m)
 =c_tP_t(\mathbf v;A_1,\ldots,A_m).
\end{equation}
In particular, $P_t(\w;aA_1,\ldots,aA_m)=aP_t(\w;\mathbf A)$ for $a>0$.

\item[(iii)] \emph{Permutation invariance.} For every permutation $\sigma$,
\begin{equation}\label{eq:permutation}
 P_t(\w_\sigma;A_{\sigma(1)},\ldots,A_{\sigma(m)})
 =P_t(\w;\mathbf A).
\end{equation}

\item[(iv)] \emph{Two-sided covariance.} For $P,Q\in\T_n$,
\begin{equation}\label{eq:two-sided-covariance}
 P_t(\w;PA_1Q,\ldots,PA_mQ)
 =P\,P_t(\w;\mathbf A)\,Q.
\end{equation}
In particular, $P_t$ is similarity equivariant under $\T_n$.

\item[(v)] \emph{Idempotency.} $P_t(\w;A,\ldots,A)=A$.

\item[(vi)] \emph{Principal-block consistency.} Equation
\eqref{eq:block-consistency} holds.

\item[(vii)] \emph{Diagonal formula.} If
$P_t(\w;\mathbf A)=[x_{pq}]$, then
\begin{equation}\label{eq:diagonal-property}
 x_{jj}=\left(\sum_{i=1}^m w_i(A_i)_{jj}^t\right)^{1/t}.
\end{equation}

\item[(viii)] \emph{Repetition invariance.} Repeating each $A_i$ $k$ times
and replacing the corresponding weights by $w_i/k$ leaves the mean
unchanged.

\item[(ix)] \emph{Reduction.} If
$\widehat\w=(1-w_m)^{-1}(w_1,\ldots,w_{m-1})$, then
\begin{equation}\label{eq:reduction}
 P_t(\w;A_1,\ldots,A_{m-1},X)=X
 \quad\Longleftrightarrow\quad
 P_t(\widehat\w;A_1,\ldots,A_{m-1})=X.
\end{equation}

\item[(x)] \emph{Scaling along $\#$.} For $t\ne 0$ and $s>0$,
\begin{equation}\label{eq:geodesic-scaling}
 P_t(\w;X\#_sA_1,\ldots,X\#_sA_m)=X
 \quad\Longleftrightarrow\quad
 P_{st}(\w;\mathbf A)=X.
\end{equation}

\end{enumerate}
\end{proposition}

\begin{proof}
  If the $A_i$ commute and
$S=\sum_iw_iA_i^t$, then $X=S^{1/t}$ commutes with every $A_i$ and
\[
 \sum_iw_i(X^{-1}A_i)^t=X^{-t}\sum_iw_iA_i^t=I_n,
\]
which proves (i) by uniqueness.

For (ii), let $Y=P_t(\mathbf v;A_1,\ldots,A_m)$ and $X=c_tY$.  Since
\[
 X\#_t(a_iA_i)=c_t^{1-t}a_i^t(Y\#_tA_i),
\]
we obtain
\[
 \sum_iw_iX\#_t(a_iA_i)
 =c_t\sum_iv_i(Y\#_tA_i)=c_tY=X.
\]
Assertion (iii) is immediate from the weighted sum.  For (iv), analytic
functional calculus is similarity equivariant, and therefore
\[
 (PXQ)\#_t(PA_iQ)
 =PXQ\bigl(Q^{-1}X^{-1}A_iQ\bigr)^t
 =P(X\#_tA_i)Q.
\]
The transformed fixed-point equation proves the covariance by uniqueness.
Assertions (v)--(vii) follow from the defining equation and
Lemma~\ref{lem:diagonal-block}.  Repetition invariance is immediate.
For (ix), use $X\#_tX=X$ and move the term $w_mX$ to the left.  Finally,
\[
 X\#_t(X\#_sA_i)=X(X^{-1}A_i)^{st}=X\#_{st}A_i,
\]
which proves (x).  
\end{proof}

A direct computation shows that 
\begin{equation}\label{eq:endpoints}
	P_1(\w;\mathbf A)=\sum_{i=1}^m w_iA_i,
	\qquad
	P_{-1}(\w;\mathbf A)=\left(\sum_{i=1}^m w_iA_i^{-1}\right)^{-1}.
\end{equation}
We will prove in Theorem~\ref{thm:limit} that 
\begin{equation}
	\lim_{t\to0}P_t(\w;\mathbf A)
	=G(\w;\mathbf A). 
\end{equation}
Thus the notation of power mean can be extended  continuously to order
zero. 

\subsection{Diagonal order}

Define the diagonal order on $\T_n$: 
\[
A\preceq_{\rm d}B
\quad\Longleftrightarrow\quad
A_{jj}\le B_{jj}\quad(1\le j\le n).
\]

\begin{corollary}\label{cor:power-mean-diagonal}
Let $t, s\in\mathbb R\setminus\{0\}$ and $t<s$. 
Then 
\[
P_t(\w;A_1,\ldots,A_m)\preceq_{\rm d}
P_s(\w;A_1,\ldots,A_m). 
\]
	Moreover,
\[
\min_i(A_i)_{jj}\le
\bigl(P_t(\w;\mathbf A)\bigr)_{jj}
\le\max_i(A_i)_{jj},\qquad j=1,\ldots,n.
\]
\end{corollary}

\begin{proof}
	By \eqref{eq:diagonal-property}, for each $j=1,\ldots,n$,
	\[
	\left(P_t(\w;\mathbf A)\right)_{jj}
	=
	\left(\sum_{i=1}^m w_i(A_i)_{jj}^t\right)^{1/t}.
	\]
The classical power-mean inequality for scalars yields
	\[\min_i(A_i)_{jj}\le 
	\left(\sum_{i=1}^m w_i(A_i)_{jj}^t\right)^{1/t}
	\leq
	\left(\sum_{i=1}^m w_i(A_i)_{jj}^s\right)^{1/s}
	\le\max_i(A_i)_{jj},
	\]
	for every $j$, which proves the claim.
\end{proof}

\begin{corollary}\label{cor:diagonal-properties}
	Let $t\in\mathbb R\setminus\{0\}$.  If $A_i\preceq_{\rm d}B_i$ for every $i$,
	then
	\[
	P_t(\w;A_1,\ldots,A_m)\preceq_{\rm d}
	P_t(\w;B_1,\ldots,B_m).
	\]
\end{corollary}

\begin{proof}
	Apply the corresponding scalar properties of the weighted power mean to
	\eqref{eq:diagonal-property}.
\end{proof}

 	The positive definite power mean satisfies several order-theoretic
	properties that have no literal analogue on $\T_n$.  In particular, the
	L\"owner monotonicity, joint concavity, arithmetic--geometric--harmonic
	matrix inequalities, and Jensen inequalities for positive linear maps rely
	on self-adjoint order.  What remains canonical on $\T_n$ is the scalar order
	of the diagonal, which yields Corollaries~	\ref{cor:power-mean-diagonal} and \ref{cor:diagonal-properties}, together with the full two-sided covariance
	\eqref{eq:two-sided-covariance}.

\section{Divided Differences and Analytic Functions on Triangular Matrices}

We recall several facts concerning analytic functions of upper triangular
matrices, including their representation by divided differences and the
Parlett recurrence.  More detailed discussions can be found in
\cite{dB05}, \cite{HuangLawson26}.

Let $f$ be a real analytic function on a domain containing the positive real axis.  Define the divided differences recursively 
on tuples of distinct entries from the domain of $f$ by
\begin{equation}\label{eq:dd}
f^{[0]}(a_1)=f(a_1),\qquad
f^{[k]}(a_1,\ldots,a_{k+1})
=\frac{f^{[k-1]}(a_1,\ldots,a_k)-f^{[k-1]}(a_2,\ldots,a_{k+1})}{a_1-a_{k+1}}.
\end{equation}

\begin{remark}\label{rmk:ext}
	The Genocchi--Hermite formula gives an integral representation of the
	$k$th divided difference of a sufficiently smooth function $f$.
	For $a_0,\ldots,a_k>0$,
	\begin{equation}\label{Genocchi--Hermite}
	f^{[k]}(a_0,\ldots,a_k)
	=
	\int_{\Delta_k}
	f^{(k)}
	\left(
	\lambda_0a_0+\cdots+\lambda_ka_k
	\right)
	\,d\lambda_1\cdots d\lambda_k,
	\end{equation}
	where
\[
\Delta_k
=
\left\{(\lambda_0,\dots,\lambda_k)\in[0,1]^{k+1}:
\lambda_0+\cdots+\lambda_k=1\right\}.
\]
	The formula remains valid when some of the nodes coincide, with the
	divided difference interpreted in the confluent sense
	\cite[Sec.~9]{dB05}.  
	If $f$ is real analytic, the Genocchi--Hermite representation also shows
	that
	\[
	(a_0,\ldots,a_k)\longmapsto f^{[k]}(a_0,\ldots,a_k)
	\]
	extends from tuples of pairwise distinct nodes to a real-analytic
	function on the full domain, including tuples with repeated nodes.
	Since the set of tuples with pairwise distinct entries is dense, this
	extension is the unique continuous extension of the ordinary divided
	difference.
\end{remark}

We next turn to the use of divided differences to define the extension of $f$ to $\T_n$. For $1\le i<j\le n$, an $(i,j)$-path is a strictly increasing sequence
\[
\gamma=(i=i_1<i_2<\cdots<i_{k+1}=j).
\]
Set $|\gamma|=k$ and
\begin{align}
f^{[\gamma]}(T)&:=f^{[k]}(t_{i_1i_1},\ldots,t_{i_{k+1}i_{k+1}}),\label{eq:f-gamma}\\
\Pi_\gamma(T)&:=t_{i_1i_2}t_{i_2i_3}\cdots t_{i_ki_{k+1}}.\label{eq:Pi-gamma}
\end{align}
Let $\Gamma_{ij}$ be the collection of all $(i,j)$-paths.

\begin{theorem}[Path formula]\label{thm:path}
If $T=[t_{ij}]$ is upper triangular and $f$ is analytic on a neighborhood of its spectrum, then
\begin{equation}\label{eq:path-formula}
(f(T))_{ii}=f(t_{ii}),\qquad
(f(T))_{ij}=\sum_{\gamma\in\Gamma_{ij}}f^{[\gamma]}(T)\Pi_\gamma(T),\quad i<j.
\end{equation}
\end{theorem}

\begin{proof}
	Set
	\[
	F=f(T)=[f_{ij}]
	\qquad\text{and}\qquad
	d_i=t_{ii}.
	\]
	Since analytic functional calculus preserves upper triangularity and
	diagonal entries,
	\[
	f_{ii}=f(d_i),\qquad i=1,\ldots,n.
	\]
	Moreover, $F$ commutes with $T$:
	\[
	TF=FT.
	\]
	Taking the $(i,j)$-entry, where $i<j$, gives
	\[
	d_i f_{ij}
	+t_{ij}f(d_j)
	+\sum_{k=i+1}^{j-1}t_{ik}f_{kj}
	=
	f(d_i)t_{ij}
	+f_{ij}d_j
	+\sum_{k=i+1}^{j-1}f_{ik}t_{kj}.
	\]
	Hence
	\begin{equation}\label{eq:Parlett-entry}
		(d_i-d_j)f_{ij}
		=
		\bigl(f(d_i)-f(d_j)\bigr)t_{ij}
		+
		\sum_{k=i+1}^{j-1}
		\bigl(f_{ik}t_{kj}-t_{ik}f_{kj}\bigr).
	\end{equation}
	
	We first assume that the diagonal entries $d_1,\ldots,d_n$ are pairwise
	distinct.  We prove \eqref{eq:path-formula} by induction on $j-i$.
	If $j=i+1$, then \eqref{eq:Parlett-entry} gives
	\[
	f_{i,i+1}
	=
	\frac{f(d_i)-f(d_{i+1})}{d_i-d_{i+1}}\,t_{i,i+1}
	=
	f^{[1]}(d_i,d_{i+1})t_{i,i+1},
	\]
	which is precisely the path formula, since there is only one path from
	$i$ to $i+1$.
	
	Now let $j-i\geq2$ and assume that the formula holds for all pairs whose
	index difference is smaller than $j-i$.  Substituting the induction
	hypothesis into \eqref{eq:Parlett-entry}, we obtain
	\begin{align}
		(d_i-d_j)f_{ij}
		&=
		\bigl(f(d_i)-f(d_j)\bigr)t_{ij}
		\nonumber\\
		&\quad+
		\sum_{k=i+1}^{j-1}
		\left[
		\left(
		\sum_{\gamma\in\Gamma_{ik}}
		f^{[\gamma]}(T)\Pi_\gamma(T)
		\right)t_{kj}
		-
		t_{ik}
		\left(
		\sum_{\eta\in\Gamma_{kj}}
		f^{[\eta]}(T)\Pi_\eta(T)
		\right)
		\right].
		\label{eq:path-induction}
	\end{align}
	
	The first term on the right-hand side produces the direct path
	$\gamma=(i,j)$ and, after division by $d_i-d_j$, has coefficient
	\[
	\frac{f(d_i)-f(d_j)}{d_i-d_j}
	=
	f^{[1]}(d_i,d_j).
	\]
	
	Consider now a path with at least two edges,
	\[
	\gamma=(i_0,i_1,\ldots,i_r),
	\qquad
	i=i_0<i_1<\cdots<i_r=j,
	\qquad r\geq2.
	\]
	Its path product is
	\[
	\Pi_\gamma(T)
	=
	t_{i_0i_1}t_{i_1i_2}\cdots t_{i_{r-1}i_r}.
	\]
	In \eqref{eq:path-induction}, this path occurs exactly twice.  The first
	occurrence comes from
	$
	f_{i,i_{r-1}}t_{i_{r-1},j}
	$
	and has coefficient
	\[
	f^{[r-1]}(d_{i_0},d_{i_1},\ldots,d_{i_{r-1}}),
	\]
	whereas the second comes from
	$
	-t_{i,i_1}f_{i_1,j}
	$
	and has coefficient
	\[
	-f^{[r-1]}(d_{i_1},\ldots,d_{i_r}).
	\]
	Thus, after division by $d_i-d_j=d_{i_0}-d_{i_r}$, the coefficient of
	$\Pi_\gamma(T)$ is
	\[
	\frac{
		f^{[r-1]}(d_{i_0},\ldots,d_{i_{r-1}})
		-
		f^{[r-1]}(d_{i_1},\ldots,d_{i_r})
	}{
		d_{i_0}-d_{i_r}
	}.
	\]
	By the recursive definition of divided differences, this is exactly
	\[
	f^{[r]}(d_{i_0},d_{i_1},\ldots,d_{i_r})
	=
	f^{[\gamma]}(T).
	\]
	Since every path from $i$ to $j$ is either the direct path $(i,j)$ or
	has at least one intermediate vertex, \eqref{eq:path-formula} follows.
	This completes the induction when the diagonal entries are pairwise
	distinct.
	
It remains to remove the assumption that the diagonal entries are distinct.
Choose diagonal perturbations
\[
T_\varepsilon
=
T+\varepsilon\diag(c_1,\ldots,c_n),
\]
with distinct $c_i$, so that $T_\varepsilon$ has pairwise distinct diagonal
entries for all sufficiently small nonzero $\varepsilon$.  The formula
therefore holds for $T_\varepsilon$.  Since
\[
T_\varepsilon\to T,
\qquad
f(T_\varepsilon)\to f(T),
\]
and divided differences of an analytic function depend continuously on
their nodes, including in the confluent case of repeated nodes (see Remark \ref{rmk:ext}), we may let
$\varepsilon\to0$ to obtain
\[
(f(T))_{ij}
=
\sum_{\gamma\in\Gamma_{ij}}
f^{[\gamma]}(T)\Pi_\gamma(T).
\]
This proves the formula in general.
\end{proof}

For the power mean, the basic scalar function is
\[
p_t(x)=x^t,\qquad x>0.
\]
Its first divided difference is
\begin{equation}\label{eq:power-dd1}
p_t^{[1]}(a,b)=
\begin{cases}
\dfrac{a^t-b^t}{a-b},&a\ne b,\\[1.2ex]
t a^{t-1},&a=b.
\end{cases}
\end{equation}
For $a,b>0$, $p_t^{[1]}(a,b)$ has the same sign as $t$ and is therefore nonzero whenever $t\ne0$.

\begin{lemma} \label{lem:perturb}
Let $T\in\T_n$, $x\in\R$, and let $f$ be analytic on a neighborhood of the spectrum of $T$.  Then
\begin{equation}\label{eq:perturb}
f(T+xE_{1n})=f(T)+x f^{[1]}(t_{11},t_{nn})E_{1n}.
\end{equation}
In particular,
\begin{equation}\label{eq:power-perturb}
(T+xE_{1n})^t=T^t+x p_t^{[1]}(t_{11},t_{nn})E_{1n}.
\end{equation}
\end{lemma}

\begin{proof}
In the path formula \eqref{eq:path-formula}, changing $t_{1n}$ can affect only the $(1,n)$-entry.  The only path involving the direct arc $(1,n)$ is the length-one path $(1,n)$, and its contribution is $f^{[1]}(t_{11},t_{nn})t_{1n}$.  This proves \eqref{eq:perturb}.
\end{proof}

As a special case of Theorem~\ref{thm:path}, the inverse of a unipotent triangular matrix has the following explicit expression.

\begin{lemma}\label{lem:inverse}
If $X=[x_{ij}]\in\T_n$ is unipotent, then
\begin{equation}\label{eq:inverse-path}
(X^{-1})_{ij}=\sum_{\gamma\in\Gamma_{ij}}(-1)^{|\gamma|}\Pi_\gamma(X),\qquad i<j.
\end{equation}
Consequently, for $A=[a_{ij}]\in\T_n$,
\begin{equation}\label{eq:XinvA}
(X^{-1}A)_{ij}=a_{ij}+\sum_{\ell=i+1}^j\left(\sum_{\gamma\in\Gamma_{i\ell}}(-1)^{|\gamma|}\Pi_\gamma(X)\right)a_{\ell j}.
\end{equation}
\end{lemma}

\section{Existence, Uniqueness, and Analyticity on $\T_n$}

In \cite{HuangLawson26} the authors have established the existence, uniqueness, and analyticity of the
Karcher mean defined on the matrix group  $\T_n$.  In this section we obtain the analogous result for
the power means defined on $\T_n$.  In both cases the recursive formulas for the logarithm obtained
via Parlett recurrence play an important role.

Define the diagonal normalization
\begin{equation}\label{eq:D}
	D_t=\diag(d_1,\ldots,d_n),\qquad
	d_p=\left(\sum_{i=1}^m w_i\, (A_i)_{pp}^t\right)^{1/t},
\end{equation}
and set
\begin{equation}\label{eq:normalized-data}
	A_{(i)}=D_t^{-1}A_i=[a_{pq}^{(i)}].
\end{equation}
Then
\begin{equation}\label{eq:diag-normalized}
	\sum_{i=1}^m w_i(a_{pp}^{(i)})^t=1,
	\qquad p=1,\ldots,n.
\end{equation}
For $p<q$, put
\begin{equation}\label{eq:c-t}
	c_{pq}(t)=\sum_{i=1}^m w_i a_{qq}^{(i)}
	p_t^{[1]}(a_{pp}^{(i)},a_{qq}^{(i)}).
\end{equation}
By \eqref{eq:power-dd1}, every summand in \eqref{eq:c-t},
and therefore $c_{pq}(t)$ itself, has the same sign as $t$.
In particular, $c_{pq}(t)\ne 0.$

\begin{theorem}[Existence, uniqueness, and analyticity]\label{thm:existence}
	For every $t\in\mathbb{R}\setminus\{0\}$, every positive weight vector $\w\in\Delta_m^\circ$, and every $A_1,\ldots,A_m\in\T_n$, equation \eqref{eq:power-mean} has a unique solution
	\[
	P_t(\w;A_1,\ldots,A_m)\in\T_n.
	\]
	For fixed $t\ne 0$, this solution $X$ depends real-analytically on the entries of $A_1,\ldots,A_m$ and on the weights in the open simplex $\Delta_m^\circ$.
\end{theorem}

\begin{proof}
	Set $X=D_tY$.  Then $X^{-1}A_i=Y^{-1}D_t^{-1}A_i=Y^{-1}A_{(i)}$, so \eqref{eq:normalized-power} for the original data is equivalent to the normalized equation for $A_{(i)}$.  By \eqref{eq:diag-normalized}, any normalized solution $Y$ satisfies $\diag(Y)=I_n$.
	
	We use induction on $n$.  The scalar case follows from \eqref{eq:diagonal}.  Assume the result on $\T_{n-1}$.  We note that the block property uniquely 
	determines the two overlapping principal blocks $Y[1:n-1]$ and $Y[2:n]$.  These two blocks agree on their common
principal block $[2 : n-1]$, as follows from the induction hypothesis and uniqueness
for the corresponding $(n-2)\times (n-2)$ normalized power equation.

	Let $Y_0$ be the unique unipotent triangular matrix having those blocks and $(Y_0)_{1n}=0$.  Then
	\[
	R_0:=\sum_{i=1}^m w_i(Y_0^{-1}A_{(i)})^t-I_n
	\]
	vanishes on both principal $(n-1)\times(n-1)$ blocks, so $R_0=rE_{1n}$.
	
	Any solution must be of the form $Y=Y_0+xE_{1n}$.  Since $Y_0$ is unipotent,
	\[
	Y^{-1}=Y_0^{-1}-xE_{1n},
	\]
	and therefore
	\[
	Y^{-1}A_{(i)}=Y_0^{-1}A_{(i)}-xa_{nn}^{(i)}E_{1n}.
	\]
	By Lemma \ref{lem:perturb},
	\[
	(Y^{-1}A_{(i)})^t=(Y_0^{-1}A_{(i)})^t
	-x a_{nn}^{(i)}p_t^{[1]}(a_{11}^{(i)},a_{nn}^{(i)})E_{1n}.
	\]
	Thus \eqref{eq:normalized-power} is equivalent to
	\[
	(r-xc_{1n}(t))E_{1n}=0.
	\]
	Because $c_{1n}(t)\ne 0$, there is exactly one choice $x=r/c_{1n}(t)$.  The construction uses analytic operations and division by a nonzero analytic function, so the resulting entries depend analytically on the data and weights.  This completes the induction.
\end{proof}

\begin{corollary}
For fixed $t\ne 0$, the map
\[
(\w,A_1,\ldots,A_m)\longmapsto P_t(\w;A_1,\ldots,A_m)
\]
is locally Lipschitz in any system of Euclidean coordinates on the open simplex times $(\T_n)^m$.
\end{corollary}

\begin{proof}
The result follows from Theorem \ref{thm:existence} and the elementary fact that a real-analytic map between finite-dimensional manifolds is locally Lipschitz on compact coordinate neighborhoods.
\end{proof}

\section{Recursive Formulas}

\subsection{Entrywise recursion}

For $p<q$, let $X'_{pq}$ denote the unipotent upper triangular matrix that agrees with the already computed entries of $X'$ in the block $[p:q]$, except that its $(p,q)$-entry is set equal to zero; entries outside that block are those of the identity.

\begin{theorem}[Entrywise recursive algorithm]\label{thm:entry-recursion}
With the notation of \eqref{eq:D}--\eqref{eq:c-t},
\[
P_t(\w;A_1,\ldots,A_m)=D_tX',
\qquad \diag(X')=I_n,
\]
and the strictly upper triangular entries of $X'=[x'_{pq}]$ are determined recursively, superdiagonal by superdiagonal, by
\begin{equation}\label{eq:entry-recursion}
x'_{pq}
=
\frac{1}{c_{pq}(t)}
\sum_{i=1}^m w_i
\sum_{\gamma\in\Gamma_{pq}}
 p_t^{[\gamma]}(A_{(i)})
\Pi_\gamma\bigl((X'_{pq})^{-1}A_{(i)}\bigr),
\qquad p<q.
\end{equation}
The entries of $(X'_{pq})^{-1}A_{(i)}$ are given explicitly by \eqref{eq:XinvA}.
\end{theorem}

\begin{proof}
Apply the induction argument in Theorem \ref{thm:existence} to the principal block $[p:q]$.  With its unknown corner entry set to zero, the $(p,q)$-entry of
\[
\sum_i w_i\bigl((X'_{pq})^{-1}A_{(i)}\bigr)^t-I
\]
is the numerator in \eqref{eq:entry-recursion}, by Theorem \ref{thm:path}.  Changing the corner entry from $0$ to $x'_{pq}$ changes this quantity by $-c_{pq}(t)x'_{pq}$, by Lemma \ref{lem:perturb}.  Solving the resulting linear equation gives \eqref{eq:entry-recursion}.
\end{proof}

We illustrate the algorithm by deriving explicit formulas on $\T_2$ and $\T_3$.

Recall the notation $p_t(x)=x^t$ and the first divided difference in \eqref{eq:power-dd1}.

\begin{theorem}\label{thm:UT2}
	For $t\ne 0$, the $\w$-weighted power mean of $A_1,\ldots,A_m\in \T_2$ exists uniquely and is given by
	\begin{equation}\label{eq:UT2}
		P_t(\w;A_1,\ldots,A_m)
		=
		\begin{bmatrix}
			\bar a_{11}&x_{12}\\0&\bar a_{22}
		\end{bmatrix},
	\end{equation}
	where 
	\begin{equation}\label{eq:bars}
		\bar a_{11}=\left(\sum_{k=1}^m w_k(A_k)_{11}^t\right)^{1/t},
		\qquad
		\bar a_{22}=\left(\sum_{k=1}^m w_k(A_k)_{22}^t\right)^{1/t},
	\end{equation}
	\begin{equation}
		A_{(k)}=\diag(\bar a_{11}, \bar a_{22})^{-1} A_k=\mtx{a_{ij}^{(k)}},\qquad k=1,2,\ldots,m,
	\end{equation}
	and 
	\begin{equation}\label{eq:UT2-x}
		x_{12}
		=\bar a_{11}
		\frac{\sum_{k=1}^m w_k \,
			p_t^{[1]}\left( a_{11}^{(k)} , a_{22}^{(k)} \right)  a_{12}^{(k)}}{\sum_{k=1}^m w_k \,
			p_t^{[1]}\left( a_{11}^{(k)} , a_{22}^{(k)} \right) a_{22}^{(k)}}.
	\end{equation}
	The entries depend analytically on the entries of $A_1,\ldots,A_m$.
\end{theorem}

\begin{proof}
	The diagonal equations in \eqref{eq:normalized-power} are exactly \eqref{eq:bars}.  By Theorem \ref{thm:path}, the $(1,2)$-entry of $(X^{-1}A_k)^t$ equals
	\[
		p_t^{[1]}\left( a_{11}^{(k)} , a_{22}^{(k)} \right) \left( a_{12}^{(k)} -
	 x_{12} \frac{a_{22}^{(k)}}{\bar a_{11}} \right).
	\]
	The $(1,2)$-entry of \eqref{eq:normalized-power} must vanish, so
	\[
	0=\sum_{k=1}^m w_k\, p_t^{[1]}\left( a_{11}^{(k)} , a_{22}^{(k)} \right) \left( a_{12}^{(k)} -
	x_{12} \frac{a_{22}^{(k)}}{\bar a_{11}} \right).
	\]
	Solving for $x_{12}$ gives \eqref{eq:UT2-x}.  All $p_t^{[1]}\left( a_{11}^{(k)} , a_{22}^{(k)} \right)$ have the same nonzero sign as $t$, 
	hence the denominator is nonzero.  Analyticity follows from analyticity of divided differences.
\end{proof}

Now consider the power means on $\T_3$. 
Let
\[
X'=\begin{bmatrix}
1&x'_{12}&x'_{13}\\
0&1&x'_{23}\\
0&0&1
\end{bmatrix}.
\]
Then
\begin{align}
x'_{12}
&=\frac{1}{c_{12}(t)}\sum_{i=1}^m w_i
p_t^{[1]}(a_{11}^{(i)},a_{22}^{(i)})a_{12}^{(i)},\label{eq:x12}\\
x'_{23}
&=\frac{1}{c_{23}(t)}\sum_{i=1}^m w_i
p_t^{[1]}(a_{22}^{(i)},a_{33}^{(i)})a_{23}^{(i)}.\label{eq:x23}
\end{align}
For the corner entry, define
\begin{align*}
u_i&:=a_{12}^{(i)}-x'_{12}a_{22}^{(i)},\\
v_i&:=a_{23}^{(i)}-x'_{23}a_{33}^{(i)},\\
z_i&:=a_{13}^{(i)}-x'_{12}a_{23}^{(i)}+x'_{12}x'_{23}a_{33}^{(i)}.
\end{align*}

\begin{theorem}\label{thm:UT3}
For $t\ne 0$ and $A_1,\ldots,A_m\in\T_3$,
\[
P_t(\w;A_1,\ldots,A_m)=D_tX',
\]
where $x'_{12}$ and $x'_{23}$ are given by \eqref{eq:x12}--\eqref{eq:x23} and
\begin{equation}\label{eq:x13}
x'_{13}
=\frac{1}{c_{13}(t)}\sum_{i=1}^m w_i
\left[
p_t^{[1]}(a_{11}^{(i)},a_{33}^{(i)})z_i
+p_t^{[2]}(a_{11}^{(i)},a_{22}^{(i)},a_{33}^{(i)})u_i v_i
\right].
\end{equation}
\end{theorem}

\begin{proof}
The formulas for the first superdiagonal follow directly from Theorem \ref{thm:entry-recursion}.  For $x'_{13}$, apply the $3\times3$ case of the path formula
\[
(T^t)_{13}=p_t^{[1]}(t_{11},t_{33})t_{13}
+p_t^{[2]}(t_{11},t_{22},t_{33})t_{12}t_{23}
\]
to $T=(X'_{13})^{-1}A_{(i)}$.  Formula \eqref{eq:XinvA} gives precisely $t_{12}=u_i$, $t_{23}=v_i$, and $t_{13}=z_i$.
\end{proof}

\subsection{A block recursive algorithm}

For $A\in\T_\ell$ and $a>0$, define
$
p_t^{[1]}(A,aI)
$
by applying analytic functional calculus to the scalar function
\[
z\longmapsto p_t^{[1]}(z,a).
\]
By the Genocchi--Hermite formula \eqref{Genocchi--Hermite}, this function is analytic on
$(0,\infty)$.
If $a\notin\sigma(A)$, then
\begin{equation}\label{eq:block-power-explicit}
	p_t^{[1]}(A,aI)
	=
	(A^t-a^tI)(A-aI)^{-1}.
\end{equation}

The block Parlett formula for $p_t(x)=x^t$ takes the following form.

\begin{lemma}\label{lem:block-power}
	Let
	\[
	T=
	\begin{bmatrix}
		T_{11}&T_{12}\\
		0&a
	\end{bmatrix}
	\in\T_n.
	\]
	Then
	\begin{equation}\label{eq:block-power}
		T^t=
		\begin{bmatrix}
			T_{11}^t & p_t^{[1]}(T_{11},aI)T_{12}\\
			0&a^t
		\end{bmatrix}.
	\end{equation}
\end{lemma}

\begin{theorem}[Block recursive algorithm]\label{thm:block-recursion}
Let $X=P_t(\w;A_1,\ldots,A_m)=D_tX'$ and $A_{(i)}=D_t^{-1}A_i$.  Define
\[
X'_{\ell}=X'[1:\ell],\qquad A_{(i,\ell)}=A_{(i)}[1:\ell].
\]
Suppose $X'_{\ell-1}$ is known.  Write
\[
X'_{\ell}=\begin{bmatrix}X'_{\ell-1}&\y'_{\ell-1}\\0&1\end{bmatrix},
\qquad
A_{(i,\ell)}=\begin{bmatrix}A_{(i,\ell-1)}&\b_{(i,\ell-1)}\\0&a_{\ell\ell}^{(i)}\end{bmatrix},
\]
and set
\begin{equation}\label{eq:G-power}
G_{i,\ell-1}^{(t)}
:=p_t^{[1]}\left((X'_{\ell-1})^{-1}A_{(i,\ell-1)},a_{\ell\ell}^{(i)}I_{\ell-1}\right).
\end{equation}
Then
\begin{equation}\label{eq:block-recursion}
\y'_{\ell-1}
=X'_{\ell-1}
\left(\sum_{i=1}^m w_i a_{\ell\ell}^{(i)}G_{i,\ell-1}^{(t)}\right)^{-1}
\left[\sum_{i=1}^m w_iG_{i,\ell-1}^{(t)}(X'_{\ell-1})^{-1}\b_{(i,\ell-1)}\right].
\end{equation}
Thus $X'$ can be computed in $(n-1)$ block-recursive stages.
\end{theorem}

\begin{proof}
The normalized power equation on the $\ell\times\ell$ block is
\[
\sum_{i=1}^m w_i\bigl((X'_{\ell})^{-1}A_{(i,\ell)}\bigr)^t=I_\ell.
\]
Now
\[
(X'_{\ell})^{-1}A_{(i,\ell)}
=\begin{bmatrix}
(X'_{\ell-1})^{-1}A_{(i,\ell-1)} &
(X'_{\ell-1})^{-1}\b_{(i,\ell-1)}-a_{\ell\ell}^{(i)}(X'_{\ell-1})^{-1}\y'_{\ell-1}\\
0&a_{\ell\ell}^{(i)}
\end{bmatrix}.
\]
Using Lemma \ref{lem:block-power}, the upper-right block of the normalized power equation is
\[
0=\sum_{i=1}^m w_iG_{i,\ell-1}^{(t)}
\left[(X'_{\ell-1})^{-1}\b_{(i,\ell-1)}-a_{\ell\ell}^{(i)}(X'_{\ell-1})^{-1}\y'_{\ell-1}\right].
\]
Solving gives \eqref{eq:block-recursion}.  The matrix to be inverted is upper triangular, and its $p$th diagonal entry equals
\[
\sum_{i=1}^m w_i a_{\ell\ell}^{(i)}p_t^{[1]}(a_{pp}^{(i)},a_{\ell\ell}^{(i)}),
\]
which is nonzero because all summands have the sign of $t$.  Hence the inverse exists.
\end{proof}

\section{The Karcher Mean as the Limit $t\to0$}

Lim and P\'alfia \cite{LimPalfia2012} proved that, on the cone of positive definite matrices, their power means converge to the Karcher mean as $t\to0$.  We now establish the corresponding   limit on $\T_n$ by exploiting the entrywise recursion developed above.

Equation \eqref{eq:normalized-power} becomes trivial at $t=0$, so the correct limiting equation is obtained by subtracting the identity and dividing by $t$:
\begin{equation}\label{eq:rescaled}
\sum_{i=1}^m w_i\frac{(X^{-1}A_i)^t-I_n}{t}=0.
\end{equation}
Since
\[
\lim_{t\to0}\frac{B^t-I_n}{t}=\log B
\qquad(B\in\T_n),
\]
we formally recover the Karcher equation:
\begin{equation}\label{eq:Karcher-limit}
\sum_{i=1}^m w_i\log(X^{-1}A_i)=0.
\end{equation}

\begin{theorem}\label{thm:limit}
Let $A_1,\ldots,A_m\in\T_n$ and let $G(\w;A_1,\ldots,A_m)$ denote the unique weighted Karcher mean satisfying \eqref{eq:Karcher-limit}.  Then
\begin{equation}\label{eq:limit}
\lim_{t\to0}P_t(\w;A_1,\ldots,A_m)
=G(\w;A_1,\ldots,A_m).
\end{equation}
Moreover, the recursive formulas for $P_t$ converge entrywise to the corresponding logarithmic divided-difference formulas for the Karcher mean.
\end{theorem}

\begin{proof}
	For each positive scalar $a$,
	\[
	\frac{a^t-1}{t}\longrightarrow \log a.
	\]
	More generally, by the Genocchi--Hermite formula
	\eqref{Genocchi--Hermite}, the divided differences satisfy
	\begin{equation}\label{eq:dd-limit}
		\frac{p_t^{[k]}(a_1,\ldots,a_{k+1})}{t}
		\longrightarrow
		\log^{[k]}(a_1,\ldots,a_{k+1})
	\end{equation}
	locally uniformly on $(0,\infty)^{k+1}$.  The diagonal factors satisfy
	\[
	\left(\sum_i w_i a_i^t\right)^{1/t}
	\longrightarrow
	\prod_i a_i^{w_i}.
	\]
	
	The entrywise recursion \eqref{eq:entry-recursion} may be written as
	\begin{equation}
		x'_{pq}
		=
		\frac{1}{c_{pq}(t)/t}
		\sum_{i=1}^m w_i
		\sum_{\gamma\in\Gamma_{pq}}
		\frac{p_t^{[\gamma]}(A_{(i)})}{t}
		\Pi_\gamma\bigl((X'_{pq})^{-1}A_{(i)}\bigr).
	\end{equation}
	As $t\to0$, the matrices $A_{(i)}$ converge to their counterparts in
	the logarithmic recursion for $G(\w;A_1,\ldots,A_m)$.  Assuming
	inductively that the entries on the preceding superdiagonals have
	converged, the corresponding matrices $X'_{pq}$ converge as well.
	Hence, by \eqref{eq:dd-limit}, every term in the numerator converges to
	the corresponding term in the logarithmic recursion.
	Moreover, by \eqref{eq:c-t} and \eqref{eq:dd-limit},
	$c_{pq}(t)/t$ converges to the corresponding logarithmic coefficient,
	which is nonzero since
	\[
	\log^{[1]}(a,b)>0
	\qquad (a,b>0).
	\]
	It follows that $x'_{pq}$ converges to the $(p,q)$-entry determined by
	the logarithmic recursion.  Starting on the diagonal and proceeding
	superdiagonal by superdiagonal proves \eqref{eq:limit} by induction.
\end{proof}

\section{A Lie--Trotter Formula and Local Asymptotics for the Power Mean}
\label{sec:Lie-Trotter}

In this section we study the local behavior of the weighted power mean
$P_t$ near the identity.  We first determine its first-order
linearization and show that it is independent of the parameter $t\neq0$:
infinitesimally, every power mean reduces to the weighted arithmetic mean
on the Lie algebra $\t_n$.  This yields a Lie--Trotter formula and shows,
in particular, that all nonzero-order power means have the same
Lie--Trotter limit for matrix powers.

We then examine the second-order behavior.  In contrast to the universal
first-order term, the second-order term depends explicitly on $t$ and
therefore distinguishes different members of the family $\{P_t\}$.
This expansion also leads to a local comparison of the power, Karcher,
and log-Euclidean means near the identity, showing in particular that the
Karcher and log-Euclidean means agree to a higher order than the power
means.  In our context the log-Euclidean mean is given by
\[
L(\mathbf{w}, A_1,\ldots,A_m):=\exp\big(\sum_{i=1}^m w_i\log A_i\bigr).
\]

\begin{theorem}[First-order expansion and Lie--Trotter formula]
	\label{thm:Lie-Trotter}
	Let
	\[
	\w=(w_1,\ldots,w_m)\in\Delta_m^\circ,
	\qquad
	t\in\mathbb R\setminus\{0\},
	\]
	and let
	\[
	A_i:(-\varepsilon,\varepsilon)\longrightarrow\T_n,
	\qquad i=1,\ldots,m,
	\]
	be differentiable at $s=0$ with
	\[
	A_i(0)=I_n,
	\qquad
	A_i'(0)=:H_i\in\t_n. 
	\]
	Then
	\begin{equation}\label{eq:first-order-expansion}
		P_t\bigl(\w;A_1(s),\ldots,A_m(s)\bigr)
		=
		I_n
		+
		s\sum_{i=1}^m w_iH_i
		+
		o(s)
		\qquad (s\to0).
	\end{equation}
	Consequently,
	\begin{equation}\label{eq:Lie-Trotter-log}
		\lim_{s\to0}
		\frac{1}{s}
		\log
		P_t\bigl(\w;A_1(s),\ldots,A_m(s)\bigr)
		=
		\sum_{i=1}^m w_iH_i,
	\end{equation}
	and
	\begin{equation}\label{eq:Lie-Trotter}
		\lim_{s\to0}
		P_t\bigl(\w;A_1(s),\ldots,A_m(s)\bigr)^{1/s}
		=
		\exp\left(
		\sum_{i=1}^m w_iH_i
		\right).
	\end{equation}
\end{theorem}

\begin{proof}
   Set
	\[
	X(s)
	=
	P_t\bigl(\w;A_1(s),\ldots,A_m(s)\bigr).
	\]
	Since
	\[
	A_1(0)=\cdots=A_m(0)=I_n,
	\]
	the idempotency of the power mean gives
	\[
	X(0)=I_n.
	\]
	For fixed $t\ne 0$, the power mean depends real-analytically on its matrix
	arguments.  Hence $X(s)$ is differentiable at $s=0$.  Put
	\[
	K:=X'(0).
	\]
	Then
	\[
	X(s)=I_n+sK+o(s).
	\]
	Likewise,
	\[
	A_i(s)=I_n+sH_i+o(s),
	\qquad i=1,\ldots,m.
	\]
	
	The  power mean satisfies the normalized equation
	\begin{equation}\label{eq:LT-normalized}
		\sum_{i=1}^m
		w_i\bigl(X(s)^{-1}A_i(s)\bigr)^t
		=
		I_n.
	\end{equation}
	Since
	\[
	X(s)^{-1}
	=
	I_n-sK+o(s),
	\]
	we obtain
	\[
	X(s)^{-1}A_i(s)
	=
	I_n+s(H_i-K)+o(s).
	\]
	
	The Fr\'echet derivative of the power map
	$
	\Phi_t(B)=B^t
	$
	at $I_n$ is
	\[
	D\Phi_t(I_n)[Z]=tZ.
	\]
	Therefore
	\[
	\bigl(X(s)^{-1}A_i(s)\bigr)^t
	=
	I_n+st(H_i-K)+o(s).
	\]
	Substitution into \eqref{eq:LT-normalized} gives
	\[
	I_n
	+
	st\sum_{i=1}^m
	w_i(H_i-K)
	+
	o(s)
	=
	I_n.
	\]
	Since $t\neq0$ and $\sum_iw_i=1$, we get
	\[
	K
	=
	\sum_{i=1}^m w_iH_i.
	\]
	Hence
	\[
	X(s)
	=
	I_n
	+
	s\sum_{i=1}^m w_iH_i
	+
	o(s),
	\]
	which proves \eqref{eq:first-order-expansion}.

	Since the Fr\'echet derivative of the logarithm at $I_n$ is the identity,
	\[
	\log(I_n+sK+o(s))
	=
	sK+o(s).
	\]
	Applying this to \eqref{eq:first-order-expansion} gives
	\[
	\frac{1}{s}
	\log
	P_t\bigl(\w;A_1(s),\ldots,A_m(s)\bigr)
	\longrightarrow K,
	\]
	which proves \eqref{eq:Lie-Trotter-log}.  Finally,
	\[
	P_t\bigl(\w;A_1(s),\ldots,A_m(s)\bigr)^{1/s}
	=
	\exp\left[
	\frac{1}{s}
	\log
	P_t\bigl(\w;A_1(s),\ldots,A_m(s)\bigr)
	\right],
	\]
	and continuity of the exponential map yields
	\[
	\lim_{s\to0}
	P_t\bigl(\w;A_1(s),\ldots,A_m(s)\bigr)^{1/s}
	=
	e^K
	=
	\exp\left(
	\sum_{i=1}^m w_iH_i
	\right).
	\qedhere 
	\]
\end{proof}

An important special case of
Theorem \ref{thm:Lie-Trotter} is the Lie--Trotter formula for matrix powers.

\begin{corollary} 
	\label{cor:Lie-Trotter-powers}
	Let
	\[
	A_1,\ldots,A_m\in\T_n,
	\qquad
	\w\in\Delta_m^\circ,
	\qquad
	t\in\mathbb R\setminus\{0\}.
	\]
	Then
	\begin{equation}\label{eq:Lie-Trotter-powers}
		\lim_{s\to0}
		P_t\bigl(
		\w;
		A_1^s,\ldots,A_m^s
		\bigr)^{1/s}
		=
		\exp\left(
		\sum_{i=1}^m w_i\log A_i
		\right).
	\end{equation}
\end{corollary}

\begin{proof}
	For each $i$, define
	\[
	A_i(s)=A_i^s=\exp(s\log A_i).
	\]
	Then
	$
	A_i(0)=I_n
	$
	and
	\[
	A_i'(0)=\log A_i.
	\]
	The conclusion follows immediately from
	Theorem~\ref{thm:Lie-Trotter}.
\end{proof}

The formula \eqref{eq:Lie-Trotter-powers} shows that
	every power
	mean $P_t$, $t\neq0$, has the same Lie--Trotter limit for matrix powers, 
	which is the $\w$-weighted log-Euclidean mean on $\T_n$.

The Lie--Trotter formula determines the first-order behavior of every power
mean at the identity.  We now compute the second-order term.  It turns out that this is the
first order at which the parameter $t$ enters.

Let 
$
H_1,\ldots,H_m\in\t_n,
$
and put
\begin{equation}\label{eq:M-V}
	M:=\sum_{i=1}^m w_iH_i,
	\qquad
	V:=\sum_{i=1}^m w_iH_i^2-M^2.
\end{equation}
Notice that
\begin{equation}\label{eq:variance-form}
	V
	=
	\sum_{i=1}^m w_i(H_i-M)^2.
\end{equation}
The diagonal entries of $V$
are the ordinary weighted variances
\[
V_{jj}
=
\sum_{i=1}^m w_i(H_i)_{jj}^2
-
\left(
\sum_{i=1}^m w_i(H_i)_{jj}
\right)^2
\geq0.
\]

\begin{theorem} 
	\label{thm:second-order-power}
	Let
	  $t\in \R \setminus\{0\}$.  Then, as $s\to0$,
	\begin{equation}\label{eq:power-second-order}
		P_t\bigl(\w;e^{sH_1},\ldots,e^{sH_m}\bigr)
		=
		I_n+sM+
		\frac{s^2}{2}\bigl(M^2+tV\bigr)
		+
		O(s^3),
	\end{equation}
	and consequently
	\begin{equation}\label{eq:log-power-second-order}
		\log
		P_t\bigl(\w;e^{sH_1},\ldots,e^{sH_m}\bigr)
		=
		sM+\frac{t s^2}{2}V+O(s^3).
	\end{equation}
\end{theorem}

\begin{proof}
The map 
\[
s\mapsto 	P_t\bigl(\w;e^{sH_1},\ldots,e^{sH_m}\bigr)
\]
is real analytic. 
	Write
	\[
	X(s)
	=
	P_t\bigl(\w;e^{sH_1},\ldots,e^{sH_m}\bigr)
	=
	I_n+sM+s^2N+O(s^3).
	\]
	The first-order Lie--Trotter formula gives the coefficient $M$.
	
	We have
	\begin{align*}
		X(s)^{-1}
		&=
		I_n-sM+s^2(M^2-N)+O(s^3),
		\\
		e^{sH_i}
		&=
		I_n+sH_i+\frac{s^2}{2}H_i^2+O(s^3).
	\end{align*}
	Hence
	\[
	X(s)^{-1}e^{sH_i}
	=
	I_n+sD_i+s^2E_i+O(s^3),
	\]
	where
	\begin{align*}
		D_i
		&=H_i-M, \\
		E_i
		&=
		M^2-N-MH_i+\frac12H_i^2.
	\end{align*}
	For an arbitrary pair $D,E\in\t_n$,
	\[
	(I_n+sD+s^2E)^t
	=
	I_n+stD
	+
	s^2\left(
	tE+\frac{t(t-1)}2D^2
	\right)
	+
	O(s^3).
	\]
	Substituting into the normalized power equation gives, at order $s^2$,
	\[
	\sum_{i=1}^m w_iE_i
	+
	\frac{t-1}{2}
	\sum_{i=1}^m w_iD_i^2
	=0.
	\]
	Now
	\begin{align*}
		\sum_{i=1}^m w_iE_i
		&=
		-N+\frac12\sum_{i=1}^m w_iH_i^2, 
		\\
		\sum_{i=1}^m w_iD_i^2
		&=
		\sum_{i=1}^m w_iH_i^2-M^2
		=
		V.
	\end{align*}
	Thus
	\[
	N
	=
	\frac12M^2+\frac t2V,
	\]
	which proves \eqref{eq:power-second-order}.
	
	Finally,
	\[
	\log(I_n+sM+s^2N)
	=
	sM+s^2\left(N-\frac12M^2\right)+O(s^3),
	\]
	giving \eqref{eq:log-power-second-order}. 
\end{proof}

We use the second-order formula to compare the power mean with the
Karcher and log-Euclidean means.

\begin{theorem}
	\label{thm:local-comparison}
	With $M$ and $V$ as in \eqref{eq:M-V}, set
	\begin{align*}
		P_t(s)
		&=
		P_t(\w;e^{sH_1},\ldots,e^{sH_m}),
		\\ 
		G(s)
		&=
		G(\w;e^{sH_1},\ldots,e^{sH_m}),
		\\
		L(s)
		&=
		L(\w;e^{sH_1},\ldots,e^{sH_m}).
	\end{align*}
	Then
	\begin{align}
		\log P_t(s)
		&=
		sM+\frac{ts^2}{2}V+O(s^3),
		\label{eq:P-local}\\
		\log G(s)
		&=
		sM+O(s^3),
		\label{eq:G-local}\\
		\log L(s)
		&=
		sM.
		\label{eq:L-local}
	\end{align}
\end{theorem}

\begin{proof}
	Equation \eqref{eq:P-local} is
	Theorem~\ref{thm:second-order-power}.  The log-Euclidean identity is
	immediate:
	\[
	L(s)
	=
	\exp\left(
	s\sum_iw_iH_i
	\right)
	=
	e^{sM}.
	\]
	
	For the Karcher mean, write
	\[
	G(s)=e^{Z(s)},
	\qquad
	Z(s)=sM+s^2Z_2+O(s^3).
	\]
	The Karcher equation becomes
	\[
	0
	=
	\sum_iw_i
	\log\bigl(e^{-Z(s)}e^{sH_i}\bigr).
	\]
	By the Baker--Campbell--Hausdorff formula,
	\[
	\log\bigl(e^{-Z(s)}e^{sH_i}\bigr)
	=
	s(H_i-M)
	-s^2Z_2
	-\frac{s^2}{2}[M,H_i]
	+O(s^3).
	\]
	After taking the weighted sum,
	\[
	\sum_iw_i[M,H_i]
	=
	[M,M]
	=
	0,
	\]
	and hence $Z_2=0$.  This proves \eqref{eq:G-local}.
\end{proof}

Theorem~\ref{thm:local-comparison} exhibits a qualitative distinction
	between the three means.  The Karcher and log-Euclidean means agree through
	second order near the identity, whereas the power mean differs from them at second order.  

The formula~\eqref{eq:G-local}   directly implies the
Lie--Trotter formula for the Karcher mean.

\begin{corollary} 
	\label{cor:Lie-Trotter-Karcher}
	Let
	$
	A_1,\ldots,A_m\in\T_n
	$
	and $\w\in\Delta_m^\circ.$
	Then
	\begin{equation}\label{eq:Lie-Trotter-Karcher}
		\lim_{s\to0}
		G\bigl(
		\w;
		A_1^s,\ldots,A_m^s
		\bigr)^{1/s}
		=
		\exp\left(
		\sum_{i=1}^m w_i\log A_i
		\right).
	\end{equation}
\end{corollary}

\section{Concluding Remarks}

The weighted power mean on $\T_n$ provides a nonlinear family of means
whose structure is governed by the analytic functional calculus of upper
triangular matrices.  The central mechanism is the divided-difference
representation of triangular matrix functions: once the lower
superdiagonals of a prospective solution have been determined, the next
unknown entry enters the normalized power equation linearly, with a
coefficient involving the first divided difference of $p_t(x)=x^t$.
Since this coefficient is nonzero for every $t\ne0$, the argument yields
existence, uniqueness, real-analytic dependence, and finite recursive
algorithms throughout the full range
$t\in\mathbb R\setminus\{0\}$.

The parameter $t$ connects several different aspects of the theory.
The limit $t\to0$ transforms the power equation into the Karcher equation
and gives
\[
P_t(\w;\mathbf A)\longrightarrow G(\w;\mathbf A).
\]
At the infinitesimal level, however, all nonzero-order power means have
the same first-order linearization at the identity, while the parameter
$t$ first appears in the second-order term.  This leads both to the
Lie--Trotter formula
\[
\lim_{s\to0}
P_t(\w;A_1^s,\ldots,A_m^s)^{1/s}
=
\exp\left(\sum_{i=1}^m w_i\log A_i\right)
\]
and to a precise local comparison with the Karcher and log-Euclidean
means.

Several directions remain open.  Although the diagonal entries inherit
the classical monotonicity of scalar power means, it would be interesting
to identify stronger comparison principles for $P_t$ that capture the
strictly upper triangular part.  
 Other natural problems include  numerical   stability of the recursive algorithms, and 
 extensions of the divided-difference to   general solvable Lie groups.


\begin{thebibliography}{99}

\bibitem{dB05}
C. de Boor, \emph{Divided Differences}, Surv. Approx. Theory \textbf{1} (2005), 46--69.

\bibitem{Hu26}
H. Huang, \emph{The Karcher mean on $2\times2$ triangular matrices with positive diagonal entries}, arXiv:2609.37018. 

\bibitem{HuangLawson26}
H. Huang and J. Lawson,
\emph{The Karcher Mean on Lie Groups of Upper Triangular Matrices},  arXiv:2609.34267.  

\bibitem{Law24}
J. Lawson, \emph{Weighted Karcher means on unipotent Lie groups}, Adv. Oper. Theory (2024).

\bibitem{Law25}
J. Lawson, \emph{The weighted exponential (Karcher) mean on Lie groups}, Acta Sci. Math. (Szeged), published online 2025.

\bibitem{LawsonLim2014}
J. Lawson and Y. Lim,
\emph{Karcher means and Karcher equations of positive definite operators},
Trans. Amer. Math. Soc. Ser. B \textbf{1} (2014), 1--22.


\bibitem{LimPalfia2012}
Y.~Lim and M.~P\'alfia,

\emph{Matrix power means and the Karcher mean},
J. Funct. Anal. \textbf{262} (2012), 1498--1514.

\bibitem{Par76}
B. Parlett, \emph{A recurrence among the elements of functions of triangular matrices}, Linear Algebra Appl. \textbf{14} (1976), 117--121.

\bibitem{Pennec}
X.~ Pennec, \emph{Manifold-valued image processing with SPD matrices}. in: \emph{Riemannian Geometric Statistics in
Medical Image Analysis}, Chap. 3, Elsevier, 2020, 75--234.

\bibitem{PennecArsigny}
X.~Pennec and V.~Arsigny,
\emph{Exponential barycenters of the canonical Cartan connection and
invariant means on Lie groups},
in: \emph{Matrix Information Geometry},
Springer, 2012, 123--166.


\end{thebibliography}
\end{document}